\documentclass[12pt]{article}

\usepackage[utf8]{inputenc}
\usepackage{amsmath, amssymb}
\usepackage{graphicx}
\usepackage{cite}
\usepackage{geometry}
\usepackage{setspace}
\usepackage{hyperref}
\usepackage{amsthm}
\usepackage{natbib}
\usepackage{xcolor}
\usepackage{float}
\usepackage{tikz}
\usepackage{bm}
\usepackage[normalem]{ulem}

\newtheorem{theorem}{Theorem}
\newtheorem{lemma}{Lemma}
\newtheorem{remark}{Remark}

\newtheorem{proposition}{Proposition}

\newcommand{\RNum}[1]{\rm \uppercase\expandafter{\romannumeral #1\relax}}

\title{Long- and short-time behavior of hypocoercive evolution equations with higher index via modal decompositions}
\author{Hannes Gernandt, Marco Roschkowski \\
\normalsize University of Wuppertal \\
\normalsize \texttt{\{gernandt, roschkowski\}@uni-wuppertal.com}}

\date{\today}

\begin{document}

\maketitle
\begin{abstract}
Hypocoercivity emerged in kinetic transport theory, allowing to derive exponential long-time estimates for evolution equations. Recently, the short-time asymptotics for 
equations with dissipative  generators were obtained using the hypocoercivity index that is in finite dimensions surprisingly given by a Kalman-type rank condition well-known in control theory. However, the situation for unbounded generators is only understood for index one if modal decompositions are available. Here, we prove long- and short-time estimates for unbounded generators with higher index admitting a modal decomposition. Additionally, an explicit Lyapunov functional is constructed. The result is applied to a class of port-Hamiltonian systems with distributed dissipation.
\end{abstract}


\section{Introduction}

We consider the long- and short-time behavior of evolution equations that are modeled as abstract Cauchy problems in a Hilbert space $\mathcal{X}$ with inner product $\langle\cdot,\cdot\rangle$ of the form
\begin{equation}\label{eqn:abstract ode}
    \frac{\mathrm{d}}{\mathrm{d}t} x(t) = \bm{A}x(t),\quad x(0)=x_0.
\end{equation}
    To analyze the dynamics of \eqref{eqn:abstract ode}, it is often assumed that $\bm{A}$ generates a strongly continuous semigroup $(\bm{T}(t))_{t\ge 0}$. Then the solutions of \eqref{eqn:abstract ode} are given by $x(t)=\bm{T}(t)x_0$. The long-time behavior of infinite-dimensional systems with dissipation has attracted a lot of attention; see for example \cite{villani2009hypocoercivity, bastin2016stability}. Often, one is interested in exponential stability, which means that there are constants $C, \omega > 0$ such that
    \begin{equation}\label{eqn: intro long}
        \|\bm{T}(t)\| \le C\, \mathrm{e}^{-\omega t}.
    \end{equation}
The exponential stability~\eqref{eqn: intro long} as a long-time asymptotic  is typically obtained from Lyapunov's direct method which amounts to finding a bounded operator $\bm{P} \in \mathcal{B}(\mathcal{X})$ that is self-adjoint and positive, i.e.\ $\langle x, \bm{P}x \rangle > 0$ for all $x \in \mathcal{X}\setminus \{0\}$, with
\begin{equation}\label{eqn: lya general}
    \exists \sigma > 0:\ \langle \bm{P}x, \bm{A}x \rangle + \langle \bm{A}x, \bm{P}x \rangle \le -\sigma \|x\|^2.
\end{equation}

The long-time behavior is also studied in the context of \emph{hypocoercivity}, where operators $\bm{A}$ which are not coercive may still lead to exponentially stable semi-groups by equipping the underlying Hilbert space with a different norm~\cite{villani2009hypocoercivity}.

The concept of hypocoercivity emerged in the field of kinetic transport theory with the aim of extending exponential stability analysis in the absence of coercivity; see for example \cite{dolbeault2015hypocoercivity, herau2006hypocoercivity, liu2019hypocoercivity, achleitner2016linear, villani2009hypocoercivity} and the references therein.

 More recently, the short-time behavior of solutions of~\eqref{eqn:abstract ode} is also being studied, which is captured by the following estimates 
    \begin{equation}\label{eqn: intro short}
        1 - c_1 \, t^a \le \|\bm{T}(t)\| \le 1 - c_2\, t^a, \quad 0<t<\tau
    \end{equation}
    for some $\tau > 0$ and $a \in2\mathbb{N}_0+1$.
    Property~\eqref{eqn: intro short} has been thoroughly studied in \cite{achleitner2023hypocoercivity} in the finite-dimensional case and in \cite{achleitner2025hypocoercivity} in the bounded infinite-dimensional case.

To further analyze the short-time behavior for bounded $\bm{C}\in \mathcal{B}(\mathcal{X})$, we define the \emph{skew-Hermitian part}  $\bm{C}_{\mathrm{S}} = \frac{\bm{C} - \bm{C}^*}{2}$ and the \emph{Hermitian part} $\bm{C}_{\mathrm{H}} = \frac{\bm{C} + \bm{C}^*}{2}$ and say that $\bm{C}$ is \emph{accretive} if $\bm{C}_{\mathrm{H}} \ge 0$. 

The short-time behavior is related to the \emph{hypocoercivity index}  which was introduced for bounded accretive operators $\bm{C}\in \mathcal{B}(\mathcal{X})$ in Hilbert spaces in~\cite{achleitner2025hypocoercivity}. 
An accretive operator $\bm{C}$ is called \emph{hypocoercive} if there is $ m \in \mathbb{N}$ such that \begin{equation}\label{eqn: hypo defi}
        \sum\limits_{j=0}^m \bm{C}^j \bm{C}_{\mathrm{H}} (\bm{C}^*)^j \ge \kappa \bm{1}_{\mathcal{X}}, \quad \kappa > 0,
    \end{equation}
where $\bm{1}_{\mathcal{X}}$ is the identity operator and \eqref{eqn: hypo defi} is understood in the Loewner sense.
The smallest $m_{\mathrm{HC}} = m \in \mathbb{N}$ such that \eqref{eqn: hypo defi} holds is called the \emph{hypocoercivity index} (or short: \emph{index}) of $\bm{C}$.
Furthermore, in \cite{achleitner2025hypocoercivity} it is shown that \eqref{eqn: hypo defi} can be replaced by \begin{equation}\label{eqn: hypo second defi}
    \sum\limits_{j=0}^m (\bm{C}^*_{\mathrm{S}})^j \bm{C}_{\mathrm{H}} \bm{C}_{\mathrm{S}}^j \ge \tilde{\kappa} \bm{1}_{\mathcal{X}}, \quad \tilde{\kappa} > 0
\end{equation}
which also yields the same hypocoercivity index.

Surprisingly, the hypocoercivity index $m_{\mathrm{HC}}$ is related to the controllability index of the pair $(\bm{C}_{\mathrm{S}}, \bm{C}_{\mathrm{H}})$ in the finite-dimensional and bounded infinite-dimensional situation, see Proposition~2.1 and Appendix~B in \cite{achleitner2025hypocoercivity}.
Moreover, in a more general Hilbert space setting, the hypocoercivity index $m_{\mathrm{HC}}$ is related to the decay-rates in the short-time estimate~\eqref{eqn: intro short} given by $a=2m_{\mathrm{HC}}+1$. 
To prove this result, there it is used that the semigroup generated by $\bm{A}=-\bm{C}$ is given by the operator exponential $\mathrm{e}^{-\bm{C}t}=\sum_{j=0}^\infty (-\bm{C})^j \, \frac{t^j}{j!}$. This represents a significant obstacle when one tries to apply the existing theory to abstract Cauchy problems~\eqref{eqn:abstract ode} arising from PDEs where $\bm{A}$ is typically unbounded.

Regarding short-time estimates and the  hypocoercivity index for Cauchy problems~\eqref{eqn: class intro} with unbounded generators $\bm{A}$ only partial results cover this situation. As a first example, the Lorentz kinetic equation was studied in~\cite{achleitner2025hypocoercivity} which could be reduced to a family of bounded equations by considering it on a torus and using modal decompositions.

As pointed out in \cite{achleitner2025long}, this approach works more generally when a modal decomposition for \eqref{eqn:abstract ode} is known.
More precisely, suppose that $\mathcal{X} =  \oplus_{\eta \in E} \mathcal{H}$ for another Hilbert space $\mathcal{H}$ and \begin{align}\label{eqn: class intro}
\begin{split}    
    \bm{A}(x_\eta)_{\eta \in E} = (-\bm{C}_\eta x_\eta)_{\eta \in E},\quad D(\bm{A})=\big\{(x_\eta)_{\eta \in E} \in \mathcal{X}\, \big|\, \sum\limits_{\eta \in E} \|\bm{C}_\eta x_\eta\|^2 < \infty\big \}
\end{split}
\end{align}
given a countable set $E \subset \mathbb{R}\setminus (-1, 1)$.
Here, for a given bounded accretive operator $\bm{C}\in\mathcal{B}(\mathcal{H})$, the operator $\bm{C}_\eta$ is given by $\bm{C}_\eta = \eta \bm{C}_{\mathrm{S}} + \bm{C}_{\mathrm{H}}$ with $\bm{C}_{\mathrm{H}} \in \mathcal{B}(\mathcal{H})$ bounded, self-adjoint and accretive meaning that $\bm{C}_{\mathrm{H}} \ge 0$, $\bm{C}_{\mathrm{S}} \in \mathcal{B}(\mathcal{H})$ bounded and skew-adjoint. To prove long- and short-time asymptotics for the operator~\eqref{eqn: class intro}, it was assumed in~\cite{achleitner2025long} that the underlying accretive operator $\bm{C}$ has hypocoercivity index 
$m_{\mathrm{HC}} = 1$ and that the following technical assumptions hold
\begin{equation}\label{eqn: achleitner technical}
    \dim\ker \bm{C}_{\mathrm{H}} < \infty,\quad \exists \gamma>0: \, \bm{C}_{\mathrm{H}} \ge \gamma \bm{1}_{\left(\ker\bm{C}_{\mathrm{H}}\right)^\perp}.
\end{equation}
In the present work, we obtain long-time and short-time asymptotics of \eqref{eqn: class intro} when $\bm{C}$ has an arbitrary hypocoercivity index and without the technical assumptions \eqref{eqn: achleitner technical}. Furthermore, we derive an explicit solution of the Lyapunov inequality $\bm{P}$ satisfying \eqref{eqn: lya general} and $\bm{P} \ge \bm{1}_{\mathcal{H}}$. Moreover, the short-time bounds can be directly proven, without relying on long-time asymptotics.

A reasonably large class of systems with unbounded semigroup generator $\bm{A}$ which allows to incorporate the dissipation in the system coefficients are port-Hamiltonian (pH) systems. These are a special class of systems which are naturally associated with an energy norm and dissipative components.
This means that asymptotics of the norm have a natural physical interpretation as asymptotic behavior of the total energy contained in the system; see \cite{van2014port,jacob2012linear}.  That makes pH~systems a natural candidate to derive estimates of the form \eqref{eqn: intro short}.

    The long-time behavior of solutions in particular in terms of exponential stability and stabilizability for this class is well-understood; see \cite{VilZLGM09,mora2023exponential,waurick2022asymptotic,ramirez2014exponential,philipp2021minimizing}. 
Regarding the short-time behavior, it has recently been shown that pH systems with boundary damping do not obey a short-time estimate of the form~\eqref{eqn: intro short}, see \cite{roschkowski2025energymethodsdistributedporthamiltonian}. Here, we show that for a particular class of pH systems with distributed damping, it is indeed possible to obtain short-time estimates of the form \eqref{eqn: intro short}.

The remainder of this paper is structured as follows. In Section~\ref{sec: main}, we state our main result. The proof of our main result is given in Section~\ref{sec: proof}. We apply our main result to a class of pH systems with distributed  damping in Section~\ref{sec: app}.
\section*{Notations}
We will let $\mathcal{B}(\mathcal{X})$ denote the bounded operators in a Hilbert space $\mathcal{X}$ with inner product $\langle\cdot,\cdot\rangle$ and identity~$\bm{1}_\mathcal{X}$. Given $\bm{C} \in \mathcal{B}(\mathcal{H})$, we define the \emph{skew-Hermitian part}  $\bm{C}_{\mathrm{S}} = \frac{\bm{C} - \bm{C}^*}{2}$, where $\bm{C}^*$ is the adjoint operator and the \emph{Hermitian part} $\bm{C}_{\mathrm{H}} = \frac{\bm{C} + \bm{C}^*}{2}$ and say that $\bm{C}$ is \emph{accretive} if $\bm{C}_H \ge 0$. The domain of an operator $\bm{A}$ is denoted by $D(\bm{A})$. 
The \emph{direct sum} of a family of Hilbert spaces $(\mathcal{H}_\eta)_{\eta \in E}$ with $E$ being a set is  denoted by $\oplus_{\eta \in E} \mathcal{H}_\eta$. If $\bm{B} \in \mathcal{B}(\mathcal{H})$, we let $\mathrm{e}^{\bm{B}t} = \sum\limits_{j=0}^\infty \frac{t^j}{j!} \, \bm{B}^j$ denote the semigroup generated by $\bm{B}$.
Given any two self-adjoint operators $\bm{B}_1,\bm{B}_2\in \mathcal{B}(\mathcal{X})$, i.e.\ $\bm{B}_1=\bm{B}_1^*$ and $\bm{B}_2=\bm{B}_2^*$, then $\bm{B}_1\geq \bm{B}_2$ denotes that $\langle\bm{B}_1x,x\rangle\geq \langle\bm{B}_2x,x\rangle$ holds for all $x\in\mathcal{H}$.
Furthermore, for a~given $\eta \in \mathbb{R}$ we use the floor function $\lfloor \eta \rfloor = \max\{ m \in \mathbb{Z} \, |\, m \le \eta\}$. 

\section{Main Result}\label{sec: main}

In this paper, we consider for accretive $\bm{C} \in \mathcal{B}(\mathcal{H})$ the family of bounded operators for $\eta\in\mathbb{R}$ and $|\eta| \ge 1$ given by  
\[
\bm{C}_\eta = \eta\, \bm{C}_{\mathrm{S}} + \bm{C}_{\mathrm{H}}.
\]
For a countable set $E \subset \mathbb{R}$, we consider the direct sum $\mathcal{X} = \oplus_{\eta \in E} \mathcal{H}$ and \begin{equation}\label{eqn: ce}
        \bm{C}_E: {\mathcal{X}}\supseteq D({\bm{C}_E})  \rightarrow \mathcal{X},\, {\bm{C}_E}(x_\eta)_{\eta \in E} := (\bm{C}_\eta x_\eta)_{\eta \in E},
\end{equation}
where $D(\bm{C}_E) = \{(x_\eta)_{\eta\in E} \in \mathcal{X} \,|\, \sum\limits_{\eta \in E} \|\bm{C}_\eta x_\eta\|^2 < \infty \}$. 
If $\bm{C} \in \mathcal{B}(\mathcal{H})$ is hypocoercive with index $m_{\mathrm{HC}} \in \mathbb{N}$ and $\eta \neq 0$, then by replacing $\bm{C}_{\mathrm{S}}$ in \eqref{eqn: hypo second defi} by $\eta \bm{C}_{\mathrm{S}}$, we see that $\bm{C}_\eta$ is also hypocoercive with index $m_{\mathrm{HC}}$.

The class of equations we will analyze is of the form
\begin{equation}\label{eqn:class}
    \frac{\mathrm{d}}{\mathrm{d}t} x(t) = \bm{A}x(t) = -\bm{C}_E x(t),
\end{equation}
where $\bm{A} = - \bm{C}_E$ generates a strongly continuous semigroup $(\bm{T}_E(t))_{t \ge 0}$ on $\mathcal{X}$, since  it is the sum of the skew-adjoint operator $(x_\eta)_{\eta\in E}\mapsto(\eta \bm{C}_{\mathrm{S}}x_\eta)_{\eta\in E}$ and the bounded self-adjoint operator $(x_\eta)_{\eta\in E}\mapsto(\bm{C}_{\mathrm{H}}x_\eta)_{\eta\in E}$; see \cite[Chapter 3]{engel2000one}.

\begin{theorem}\label{thm:theorem}
 Let $\bm{C} \in \mathcal{B}(\mathcal{H})$ be hypocoercive  with index $m_{\mathrm{HC}}$ on a Hilbert space $\mathcal{H}$ and let $E \subset \mathbb{R}\setminus (-1,1)$ be non-empty and countable, consider the operator $\bm{C}_E$ given by~\eqref{eqn: ce} and the semigroup $(\bm{T}_E(t))_{t \ge 0}$ generated by $-\bm{C}_E$. Then the following holds:
    \begin{itemize}
        \item [\rm (a)]  The operator $\bm{P}_E$ defined by  \begin{equation}\label{lyapunov}
            \bm{P}_E\, x := \left(\sum\limits_{j = 0}^{m_\mathrm{HC}} \frac{1}{\eta^{2j}} (\bm{C}_\eta^*)^j\bm{C}_\eta^j\, x_\eta\right)_{\eta \in E}
        \end{equation}
        for $x = (x_\eta)_{\eta \in E} \in \mathcal{X}$ is bounded and self-adjoint, $\bm{P}_E \ge \bm{1}_{\mathcal{X}}$ and there exists $\sigma > 0$ such that 
         \[
        \langle -\bm{C}_E \, x, \bm{P}_E\, x \rangle + \langle \bm{P}_E \, x, -\bm{C}_E\, x \rangle
        \le -\sigma\, \|x\|^2
        \]
        for all $x \in D(\bm{C}_E)$. Hence, the semigroup $(\bm{T}_E(t))_{t \ge 0}$ is  exponentially stable.
        \item [\rm (b)] There are constants $c_1, c_2, \tau > 0$ and  $a = 2 m_{\mathrm{HC}} + 1$ such that
        \begin{align}
            \label{short_time}
        1 - c_1 t^a \le \|\bm{T}_E(t)\|^2 \le 1 - c_2 t^a, \quad t \le \tau.
                \end{align}
    \end{itemize}
\end{theorem}
\begin{remark}
    The construction of $\bm{P}_E$ in Theorem~\ref{thm:theorem}~(a) is based on \cite{achleitner2025long}. There, it is noted for hypocoercive $\bm{C} \in \mathcal{B}(\mathcal{H})$ with index $m_{\mathrm{HC}}$ that the Lyapunov inequality~\eqref{eqn: lya general} is solved by 
    \[\bm{P} = \sum\limits_{j=0}^{m_{\mathrm{HC}}} (\bm{C}^*)^j \bm{C}^j\geq \bm{1}_{\mathcal{H}}.\] 
    
    We also point out that the proof of Theorem~\ref{thm:theorem} is constructive in the sense that the constants $\sigma, \tau, c_1, c_2>0$ can explicitly be determined in concrete situations.
\end{remark}
\section{Proof of the main result}\label{sec: proof}

First, we prove the following technical lemma which will be helpful in both deducing the long-time and the short-time asymptotic.
\begin{lemma}\label{lem: technical 1}
  Let $\bm{C} \in \mathcal{B}(\mathcal{H})$ be accretive and hypocoercive with index $m_{\mathrm{HC}} \in \mathbb{N}$. 
    Then the following statements hold.
    \begin{itemize}
        \item [\rm (a)] There exists $K>0$ such that for all $\ell \in \{1, \dots, m_{\mathrm{HC}}\}$ and all $|\eta| \ge 1$ it holds that 
        \begin{align*}
           \biggl\| (\bm{C}_\eta^*)^\ell \bm{C}_{\mathrm{H}} \bm{C}_\eta^\ell-  \eta^{2\ell}\, (\bm{C}_{\mathrm{S}}^*)^\ell \bm{C}_{\mathrm{H}} \bm{C}_{\mathrm{S}}^\ell\biggr\| \le K\, |\eta|^{2\ell - 1}.  \end{align*}
        \item [\rm (b)] There exists $\kappa_1>0$ such that for all $y \in \mathcal{H}$, there is $\ell_y \in \{0, \dots, m_{\mathrm{HC}}\}$ with
        \[
        \left\langle \bm{C}_{\mathrm{S}}^{\ell_y} y, \bm{C}_{\mathrm{H}}\bm{C}_{\mathrm{S}}^{\ell_y} y \right\rangle
        \ge \kappa_1 \|y\|^2.
        \]
        \item [\rm (c)] There exists $R\ge1$ such that for all $y \in \mathcal{H}$ and with $\ell_y$ chosen as in item \rm {(b)} it holds for all $|\eta| \ge R$
        \[
        \left\langle \bm{C}_\eta^{\ell_y} y, \bm{C}_{\mathrm{H}} \bm{C}_\eta^{\ell_y}y \right\rangle \ge \frac{1}{2}\, \eta^{2\ell_y}\left\langle \bm{C}_{\mathrm{S}}^{\ell_y} y, \bm{C}_{\mathrm{H}}\bm{C}_{\mathrm{S}}^{\ell_y} y \right\rangle .
        \]
        \item [\rm (d)] For all $R\geq 1$ there exists $\mu_R>0$ such that for all $y \in \mathcal{H}$,
        \[
        \inf\limits_{1 \le |\eta| \le R}\sum\limits_{j=0}^{m_{\mathrm{HC}}}\frac{1}{\eta^{2j}} \left\langle \bm{C}_\eta^{j} y, \bm{C}_{\mathrm{H}} \bm{C}_\eta^{j}y \right\rangle \ge \mu_R \, \|y\|^2.
        \]
        
    \end{itemize}
\end{lemma}
\begin{proof}
    To prove assertion (a), note that
    \[
    (\bm{C}_\eta^*)^\ell \bm{C}_{\mathrm{H}} \bm{C}_\eta^\ell =  (\eta \, \bm{C}_{\mathrm{S}}^*)^\ell \bm{C}_{\mathrm{H}} (\eta \, \bm{C}_{\mathrm{S}})^\ell + \sum\limits_{i=0}^{2\ell-1}\eta^i \bm{B}_i
    \]
    for some bounded $\bm{B}_i \in \mathcal{B}(\mathcal{H})$. Moreover, \[
   \frac{1}{\eta^{2\ell}} \biggl\| \sum\limits_{i=0}^{2\ell-1}\eta^i \bm{B}_i \biggr\| \le \frac{2\, m_{\mathrm{HC}}}{|\eta|} \cdot \max\limits_{i\in \{ 0, \dots, 2\ell-1 \}} \|\bm{B}_i\|
    \]
    which implies there is $K>0$ satisfying the bound in \rm (a).
    
    To prove {\rm (b)}, we use a trivial relation between the sum and the maximal summand together with the hypocoercivity of $\bm{C}$ which implies with \eqref{eqn: hypo second defi} that there is $\kappa > 0$ such that 
    \begin{align*}
        \max\limits_{\ell = 0, \dots, m_{\mathrm{HC}}} \left\langle \bm{C}_{\mathrm{S}}^\ell y, \bm{C}_{\mathrm{H}}\bm{C}_{\mathrm{S}}^\ell y \right\rangle\geq \tfrac{1}{m_{\mathrm{HC}} + 1}\sum\limits_{j=0}^{m_{\mathrm{HC}}} \left\langle \bm{C}_{\mathrm{S}}^j y, \bm{C}_{\mathrm{H}}\bm{C}_{\mathrm{S}}^j y \right\rangle\geq \frac{\kappa}{m_{\mathrm{HC}} + 1}\|y\|^2=:\kappa_1 \|y\|^2.
    \end{align*}
    To prove (c), we estimate with (a) and (b) \begin{align*}
            \left\langle \bm{C}_\eta^{\ell_y} y, \bm{C}_{\mathrm{H}} \bm{C}_\eta^{\ell_y}y \right\rangle &\ge \eta^{2\ell_y} \, \left\langle \bm{C}_{\mathrm{S}}^{\ell_y} y, \bm{C}_{\mathrm{H}}\bm{C}_{\mathrm{S}}^{\ell_y} y \right\rangle \, - K |\eta|^{2\ell_y - 1} \, \|y\|^2 \\
            &\ge \left(\eta^{2\ell_y} -  \frac{K} {\kappa_1} \, |\eta|^{2\ell_y - 1} \right) \, \left\langle \bm{C}_{\mathrm{S}}^{\ell_y} y, \bm{C}_{\mathrm{H}}\bm{C}_{\mathrm{S}}^{\ell_y} y \right\rangle \\
            & \ge \frac{1}{2} \, \eta^{2\ell_y} \, \left\langle \bm{C}_{\mathrm{S}}^{\ell_y}y, \bm{C}_{\mathrm{H}}\bm{C}_{\mathrm{S}}^{\ell_y} y \right\rangle 
        \end{align*}
        for $|\eta| \ge R:= \max\left(\frac{2 K}{\kappa_1}, 1\right) \ge 1$. 
        To prove item (d), assume that $R \ge 1$ is arbitrary and define \[
        \kappa(\eta) := \sup\left\{ \kappa \ge 0 \, \bigg|\, \sum\limits_{j=0}^{m_{\mathrm{HC}}} (\bm{C}_\eta^*)^j \bm{C}_{\mathrm{H}}  \bm{C}_\eta^j \ge \kappa \, \bm{1}_{\mathcal{H}} \right\}
        \]
        for $1 \le |\eta| \le R$. It is clear that $\eta \mapsto \kappa(\eta)$ defines a~continuous map on the compact set $\{\eta \in \mathbb{R} \, |\, 1\le |\eta| \le R\}$ with $\kappa(\eta) > 0$ for each $1 \le |\eta| \le R$. To conclude the proof, set \[\mu_R := \frac{1}{R^{2m_{\mathrm{HC}}}} \inf\limits_{1\le |\eta| \le R} \kappa(\eta) > 0\] where we used $\eta^{2j} \le R^{2m_{\mathrm{HC}}}$ for $|\eta| \le R$ since $R \ge 1$.
\end{proof}

\subsection{Proof Theorem~\ref{thm:theorem}~(a)}
    Note that $\|\bm{C}_\eta\| = \|\bm{C}_{\mathrm{H}}\| + |\eta| \|\bm{C}_{\mathrm{S}}\| \le L \, |\eta|$ with the constant $L:=\|\bm{C}_{\mathrm{H}}\| +\|\bm{C}_{\mathrm{S}}\| $. It follows that \[ \biggl\|\sum\limits_{j=0}^{m_{\mathrm{HC}}}\frac{1}{\eta^{2j}} (\bm{C}_\eta^*)^j \bm{C}_\eta^j\biggr\| \le (m_{\mathrm{HC}} + 1) \max(1, L^{2\,m_{ \mathrm{HC}}}),
    \]
 which implies that $\bm{P}_E$ is bounded. It is also clear that $\bm{P}_E$ is self-adjoint. 
    To see that $\bm{P}_E$ is bounded from below, note that \[\sum\limits_{j = 0}^{m_\mathrm{HC}} \frac{1}{\eta^{2j}} \langle x, (\bm{C}_\eta^*)^j\bm{C}^j_\eta \, x \rangle \ge \|\bm{C}^0_\eta x_\eta\|^2 = \|x_\eta\|^2\] for $x\in \mathcal{X}$
    which implies $\bm{P}_E \ge \bm{1}_{\mathcal{H}}$.

    Now fix $\kappa_1 > 0$ as in item \rm (b), $R \ge 1$ as in item \rm (c) and $\mu_R > 0$ as in item \rm (d) of Lemma~\ref{lem: technical 1}.
    Given $x \in D(\bm{P}_E)$, \begin{align*}
            \langle -\bm{C}_E \, x, \bm{P}_E \, x \rangle + \langle \bm{P}_E \, x, -\bm{C}_E \, x \rangle& = \sum\limits_{\eta \in E} \sum\limits_{j=0}^{m_{\mathrm{HC}}} \frac{1}{\eta^{2j}} \left\langle \bm{C}_\eta^j x_\eta, -(\bm{C}_\eta^* + \bm{C}_\eta) \bm{C}_\eta^j x_\eta \right \rangle\\
            & = -\sum\limits_{\eta \in E} \sum\limits_{j=0}^{m_{\mathrm{HC}}} \frac{2}{\eta^{2j}} \left\langle \bm{C}_\eta^j x_\eta, \bm{C}_{\mathrm{H}} \bm{C}_\eta^j x_\eta \right \rangle.
    \end{align*}
    Splitting this sum into terms where $|\eta| \ge R$ and terms where $|\eta| < R$, we obtain
    \begin{equation}\label{eqn: long tech 1}
    \begin{split}
    \langle -\bm{C}_E \, x, \bm{P}_E \, x \rangle + \langle \bm{P}_E \, x, -\bm{C}_E \, x \rangle&  \le 
    -2\, \mu_R \, \sum\limits_{\eta \in E, |\eta|< R}  \|x_\eta\|^2  - \sum\limits_{\eta \in E, |\eta|\ge R} \left\langle \bm{C}_{\mathrm{S}}^{\ell_{x_\eta}} x_\eta, \bm{C}_{\mathrm{H}}\bm{C}_{\mathrm{S}}^{\ell_{x_\eta}} x_\eta \right\rangle \\
    &=  -2\, \mu_R \, \sum\limits_{\eta \in E, |\eta|< R}  \|x_\eta\|^2 - \kappa_1 \, \sum\limits_{\eta \in E, |\eta|\ge R}  \|x_\eta\|^2 \\
    & \le\qquad -\sigma\,  \|x\|_{\mathcal{X}}^2, 
    \end{split}
    \end{equation}
with $\sigma := \min(\kappa_1, 2\, \mu_R)$,  where we used items (c) and (d) of Lemma~\ref{lem: technical 1} in the first estimate and item (b) in the second estimate. This shows that $(\bm{T}_E(t))_{t \ge 0}$ is exponentially stable as its generator $-\bm{C}_E$ satisfies a Lyapunov inequality, see \cite[Theorem 8.1.3]{jacob2012linear}.

\subsection{Proof of the upper bound in Theorem~\ref{thm:theorem}~(b)}
\begin{lemma}\label{prop: many bounds}
Let $\bm{C} \in \mathcal{B}(\mathcal{H})$ be hypocoercive with index $m_{\mathrm{HC}} \in \mathbb{N}$ and let $a = 2m_{\mathrm{HC}}+1$. Then there exist $k_2, t_2 > 0$ independent of $\eta$ such that
    \[
    \bigr\|\mathrm{e}^{- \bm{C}_\eta t} \bigl\|^2 \le 1 - k_2 \,|\eta|^{a-1}\, t^a, \quad \text{for all $t\le \frac{t_2}{|\eta|}$.}
    \]
\end{lemma}
\begin{proof}
Fix $\kappa_1 > 0$ as in item \rm (b), $R \ge 1$ as in item \rm (c) and $\mu_R > 0$ as in item \rm (d) of Lemma~\ref{lem: technical 1} and $y \in \mathcal{H}$ with $\|y\| = 1$.
Then Lemma~4.7 of \cite{achleitner2025hypocoercivity} with $\bm{U} = - \bm{C}^*_\eta, \bm{V} = - 2\, \bm{C}_{\mathrm{H}}, \bm{W} = -\bm{C}_\eta$ 
implies for all $\ell \in \{1, \dots, m_{\mathrm{HC}}\}$ that
\begin{align}
\nonumber
    \frac{\bigr\|\mathrm{e}^{- \bm{C}_\eta t} y\bigl\|^2 - 1}{2}
    &= - \sum_{j=0}^{\ell-1} \frac{t^{2j+1}}{(2j+1)!}\,\frac{1}{\binom{2j}{j}}
       \left\langle \bm{V}_{\eta,j} y,\; \bm{C}_{\mathrm{H}}\,\bm{V}_{\eta,j} y \right\rangle \\
    &\quad - \frac{t^{2\ell+1}}{(2\ell+1)!}\,\binom{2\ell}{\ell}\,
       \Delta^{(\ell)}_{2\ell+1,\ell}
       \left\langle \bm{C}_\eta^{\ell} y,\; \bm{C}_{\mathrm{H}}\,\bm{C}_\eta^{\ell} y \right\rangle \label{second term} \\
    &\quad + \sum_{j=2\ell+2}^{\infty} \frac{t^{j}}{j!}
       \sum_{k=\ell}^{j-\ell-1} \binom{j-1}{k}\,\Delta^{(\ell)}_{j,k}  \cdot \left\langle \big(-\bm{C}_\eta\big)^{k} y,\; \bm{C}_{\mathrm{H}}\,\big(-\bm{C}_\eta\big)^{\,j-k-1} y \right\rangle ,\label{eqn:magic formula}
\end{align}
 where $\bm{V}_{\eta, j} = \sum\limits_{k=0}^\infty  \frac{(2j+1)!}{(k+2j+1)!}\binom{k+j}{j}t^k (-\bm{C}_\eta)^{k+j}$  and $\Delta^{(\ell)}_{j,k} =  \frac{\binom{k}{\ell} \, \binom{j-k-1}{\ell}}{\binom{k+\ell}{\ell} \binom{j-k-1+\ell}{\ell}}.$
 Formula \eqref{eqn:magic formula} is also true for $\ell = 0$ by the identity (A.6) in \cite{achleitner2023hypocoercivity} which we repeat here:
 \begin{align*}
     \sum\limits_{k=0}^j &\binom{j}{k}(-\bm{C}_\eta)^k(-\bm{C}_\eta)^{j-k} =  2 \sum\limits_{k=0}^{j-1} (-\bm{C}_\eta)^k\bm{C}_{\mathrm{H}}(-\bm{C}_\eta)^{j-k-1}.
 \end{align*}
We must be careful because $\bigl\langle \bm{C}_\eta^\ell y, \bm{C}_{\mathrm{H}} \bm{C}_\eta^\ell y\bigr \rangle$ in \eqref{second term} is only expected to become large for some $\ell$. So we need to estimate \eqref{eqn:magic formula} for all $\ell$, then combine the two terms for $\ell$ making $\bigl\langle \bm{C}_\eta^\ell y, \bm{C}_{\mathrm{H}} \bm{C}_\eta^\ell y\bigr \rangle$ large and then give a combined estimate independent of the choice of $\ell$.

As a first step, using that $\Delta^{(\ell)}_{j,k} \le 1$ implies \[\sum\limits_{k=\ell}^{j-\ell-1} \binom{j-1}{k}\Delta^{(\ell)}_{j,k} \le \sum\limits_{k=0}^{j-1} \binom{j-1}{k}\le 2^{j-1},\] we can estimate the last term in \eqref{eqn:magic formula} as follows: \begin{align}
\nonumber
    &\sum\limits_{j=2\ell+2}^\infty \frac{t^j}{j!} \sum\limits_{k=\ell}^{j-\ell-1} \binom{j-1}{k} \Delta^{(\ell)}_{j, k} \left\langle \left(-\bm{C}_\eta\right)^k y, \bm{C}_{\mathrm{H}} \left(-\bm{C}_\eta\right)^{j-k-1}y  \right\rangle \\
    &\qquad\le \sum\limits_{j=2\ell+2}^\infty \frac{t^j}{j!} \sum\limits_{k=\ell}^{j-\ell-1} \binom{j-1}{k} \Delta^{(\ell)}_{j, k} \|\bm{C}_\eta\|^k  \|\bm{C}_{\mathrm{H}}\| \, \|\bm{C}_\eta\|^{j-1-k}  \|y\|^2 \nonumber\\
    &\qquad \le \sum\limits_{j=2\ell+2}^\infty \frac{t^j}{j!} 2^{j-1} |\eta|^{j-1}L^{j-1} \|\bm{C}_{\mathrm{H}}\| =: R_\eta(t) \label{eqn: up tech 1}
\end{align}
 where we used $\|\bm{C}_\eta\| \le |\eta| \, L$ with some $L>0$ and $\|y\| = 1$.
 To estimate the last term in \eqref{eqn: up tech 1}, we define $f(s) := \sum\limits_{j=2\ell+2}^\infty \frac{s^j}{j!}$ and note that $f(s) \le \mathrm{e}^s \le \mathrm{e}^{2L} =: \tilde{C}$ whenever $0 < s \le 2L$. Moreover, we estimate \begin{equation}\label{eqn: sub-trick-tech}
     f(s) = \sum\limits_{j=2\ell+2}^\infty \left(\frac{s}{2L}\right)^j \frac{(2L)^j}{j!} \le \left(\frac{s}{2L}\right)^{2\ell+2} \tilde{C}
 \end{equation}
for $0 \le s \le 2L$.
Using $R_\eta(t) = \frac{1}{2L|\eta|}f(2L|\eta|t) \, \|\bm{C}_{\mathrm{H}}\|$ and plugging into \eqref{eqn: sub-trick-tech}, we find $C>0$ independent of $\eta$ such that \begin{equation}\label{eqn: sub-trick}
    R_\eta(t) \le C t^{2\ell+2} |\eta|^{2\ell+1}, \quad t \le \frac{1}{|\eta|}.
\end{equation}
To estimate the second term~\eqref{second term}, we will distinguish two cases. \emph{Case 1:} If $|\eta| \ge R$, then we choose $\ell = \ell_y$ as defined in Lemma~\ref{lem: technical 1} and estimate the second term  \eqref{second term} as follows:
\begin{align}\label{eqn: step 1}
\begin{split}
        &\frac{t^{2\ell_y+1}}{(2\ell_y+1)!}\,\binom{2\ell_y}{\ell_y}\,
       \Delta^{(\ell_y)}_{2\ell_y+1,\ell_y}
       \left\langle \bm{C}_\eta^{\ell_y} y,\; \bm{C}_{\mathrm{H}}\, \bm{C}_\eta^{\ell_y} y \right\rangle \\
       &\qquad\ge \frac{\eta^{2\ell_y}t^{2\ell_y+1}}{2(2\ell_y+1)!}\,\binom{2\ell_y}{\ell_y}\,
       \Delta^{(\ell_y)}_{2\ell_y+1,\ell_y}
       \left\langle \bm{C}_ {\mathrm{S}}^{\ell_y} y,\; \bm{C}_{\mathrm{H}}\,\bm{C}_{\mathrm{S}}^{\ell_y} y \right\rangle \ge \eta^{2\ell_y}\, t^{2\,\ell_y+1}\, \kappa_2
       \end{split}
\end{align}
with ${\kappa}_2 =  \frac{\kappa_1}{2(2\ell_y+1)!}\,\binom{2\ell_y}{\ell_y}\, \Delta^{(\ell_y)}_{2\ell_y+1,\ell_y}$.

\emph{Case 2:} If $|\eta| \le R$, then Lemma~\ref{lem: technical 1} \rm (d) implies that there is $\ell \in \{0, \dots, m_{\mathrm{HC}}\}$ with
$\left\langle \bm{C}_\eta^\ell\, y, \bm{C}_{\mathrm{H}}\, \bm{C}_\eta^\ell\, y \right\rangle \ge \eta^{2\ell} \, \frac{\mu_R}{m_{\mathrm{HC}}+1}$. It follows that
\begin{equation}\label{eqn: step 2}
    \begin{split}
        &\frac{t^{2\ell+1}}{(2\ell+1)!}\,\binom{2\ell}{\ell}\,
       \Delta^{(\ell)}_{2\ell+1,\ell}
       \left\langle \bm{C}_\eta^{\ell} y,\; \bm{C}_{\mathrm{H}}\, \bm{C}_\eta^{\ell} y \right\rangle  \ge \eta^{2\ell}\, t^{2\,\ell+1}\, {\kappa}_3
    \end{split}
\end{equation}
where ${\kappa}_3 = \frac{\mu_R}{(m_{\mathrm{HC}}+1)\, (2\ell+1)!}\,\binom{2\ell}{\ell}\Delta^{(\ell)}_{2\ell+1,\ell}$.

Combining \eqref{eqn: up tech 1}, \eqref{eqn: sub-trick}, \eqref{eqn: step 1} and \eqref{eqn: step 2}, \eqref{eqn:magic formula} yields
\begin{equation}\label{eqn: up tech 2}
    \begin{split}
       \bigr\|\mathrm{e}^{- \bm{C}_\eta t} y\bigl\|^2 - 1 \le -{\kappa}_4 \eta^{2\ell} t^{2\,\ell+1}  + 2\, C |\eta|^{2\, \ell+1}\, t^{2\ell + 2}
    \end{split}
\end{equation}
for $t\le \frac{1}{|\eta|}$, ${\kappa}_4 = \min(2\, {\kappa}_2, 2\, {\kappa}_3) > 0$ (with $\ell$ depending on $y$ and $C, {\kappa}_4$ independently of $y$).
It follows that \begin{equation}
    \bigr\|\mathrm{e}^{- \bm{C}_\eta t}y \bigl\|^2 - 1 \le -k_2' \eta^{2 \ell} \, t^{2 \ell + 1}, \quad t \le \frac{t_2}{|\eta|}
\end{equation}
where $k_2' = \frac{\kappa_4}{2}$ and $t_2>0$ is chosen is so small that $2\, C\, t_2 \le \frac{\kappa_4}{2}$. We then obtain that
\begin{align*}
    \bigr\|\mathrm{e}^{- \bm{C}_\eta t} y\bigl\|^2 - 1 \le -k_2' \eta^{2 \ell} \, t^{2 \ell + 1}
     &\le - \frac{k_2'}{t_2^{2\, (m_{\mathrm{HC}} - \ell)}} \, |\eta|^{a-1}\, t^a = - k_2 \, |\eta|^{a-1}\, t^a, \quad t \le \frac{t_2}{|\eta|}
\end{align*}
where $k_2 = \frac{k_2'}{t_2^{2m_{\mathrm{HC}}}} = \min\bigl\{  \frac{k_2'}{t_2^{2j}} \, \big|\, j = 0, \dots, m_{\mathrm{HC}} \bigr\}$ which concludes the proof.
\end{proof}

We have established a bound for each mode in Lemma~\ref{prop: many bounds}. However, since $\frac{t_2}{|\eta|}$ approaches $0$ as $|\eta| \rightarrow \infty$, this does not immediately yield a short-time asymptotic for $\|\bm{T}_E(t)\|$, because the time-scale asymptotically vanishes for high modes. But by iteratively repeating the estimate we have, the next lemma will allow us to compensate for this by exploiting that the rate increases by a factor $|\eta|^{a-1}$. Using this simple argument, we are also able to obtain the short-time asymptotic without using an explicit long-time asymptotic constructed before-hand, as is done in \cite{achleitner2025long} which is very desirable.
\begin{lemma}\label{lem: technical 2}
    Assume that $(\bm{T}(t))_{t \ge 0}$ is a~strongly continuous semigroup and $|\eta| \ge 1$ such that
    \begin{equation}\label{eqn: technical 2 4}
            \left\|\bm{T}(t)\right\|^2 \le 1 - k_2\, |\eta|^{a-1} t^a, \quad t\le \frac{t_2}{|\eta|} 
    \end{equation}
    for $a \in 2\mathbb{N}_0+1$ and constants $k_2, t_2 > 0$ independent of $\eta$.
    Then
    \[
    \|\bm{T}(t)\|^2 \le 1 - c_2\, t^a, \quad t \le \tau_2 . 
    \]
    for constants $c_2, \tau_2  > 0$ that explicitly depend on $k_2$ and $t_2$.
\end{lemma}
\begin{proof}
    Step 1: Assume that $n = |\eta| \in \mathbb{N}$, fix $0< t \le t_2$ and use the identity $\|\bm{T}(t+s)\| \le \|\bm{T}(t)\| \, \|\bm{T}(s)\|$ iteratively to estimate
    \begin{equation}\label{eqn: technical 2 1}
    \begin{split}
        \|\bm{T}(t)\|^2 &\le \left[\left\|\bm{T}\left(\tfrac{t}{n}\right)\right\|^2\right]^{n}= \left( 1 -k_2 \,  \tfrac{t^a}{n} \right)^{n}.
    \end{split}
    \end{equation}
    It is well-known that $0 \le 1 - s\le \mathrm{e}^{-s}$ for $s \in [0,1]$. Plugging in $s = \frac{\sigma}{n}$ for $\sigma \in [0, n]$ and taking the $n$th power, this implies
    $\left(1 - \frac{\sigma}{n}\right)^n \le \mathrm{e}^{-\sigma}$. Note that $k_2\, t_2^a \le n$ automatically holds since $\|\bm{T}(t)\|$ would be negative otherwise. Hence, we may plug $\sigma = k_2\, t^a$ into \eqref{eqn: technical 2 1} and get
    \begin{equation}\label{eqn: technical 2 2}
        \|\bm{T}(t)\|^2 \le \mathrm{e}^{-k_2\, t^a}, \quad t \le t_2.
    \end{equation}
    Step 2: Assume that $|\eta| \ge 1$ : Then  \eqref{eqn: technical 2 4} implies 
    \begin{equation*}
        \left\|\bm{T}(t)\right\|^2 \le 1 - k_2\, \lfloor|\eta|\rfloor^{a-1} t^a, \quad t\le \frac{\tau_2'}{ \lfloor |\eta| \rfloor},
    \end{equation*}
    with $\tau'_2 = \frac{t_2}{2}$ where we used the trivial identity $2\, \lfloor |\eta| \rfloor \ge |\eta|$ if $|\eta| \ge 1$. Inequality~\eqref{eqn: technical 2 2} derived in Step~1 applied with $\lfloor|\eta|\rfloor$ and plugged in for $n$ yields
    \[
    \|\bm{T}(t)\|^2 \le \mathrm{e}^{-k_2\, t^a}, \quad t \le \tau_2'.
    \]
    Finally, the claim follows since $\mathrm{e}^{-k_2\, t^a} \le 1 - k_2 \, t^a + M t^{2a}$ for $t \le \tau'_2$ and some $M>0$. Choosing $c_2 = \frac{k_2}{2}$ and $\tau_2 > 0$ so small that $M\tau_2^a \le \frac{k_2}{2}$ yields the claim.
\end{proof}

\subsection{Proof of Theorem~\ref{thm:theorem}~(b)}

\begin{proposition}\label{prop: thm b}
    Let $\bm{C} \in \mathcal{B}(\mathcal{H})$ be hypocoercive with index $m_{\mathrm{HC}} \in \mathbb{N}$ and let  $E\subset \mathbb{R}\setminus{(-1,1)}$ be non-empty and countable.
    Then the following holds:
    \begin{itemize}
        \item [\rm (a)] There are $c_2, \tau_2 > 0$ such that
    \[
    \|\bm{T}_E(t)\|^2 \le 1 - c_2 \, t^a, \quad 0 \le t\le \tau_2.
    \]
    \item [\rm (b)] There are $c_1, \tau_1 > 0$ such that
    \[
    1 - c_1 \, t^a \le \|\bm{T}_E(t)\|^2, \quad 0\le t\le \tau_1.
    \]
    \end{itemize}
\end{proposition}
\begin{proof}
    First considering \rm (a), we have already seen in Lemma~\ref{prop: many bounds} that there are $k_2, t_2>0$ such that
    \[
    \bigr\|\mathrm{e}^{- \bm{C}_\eta t} \bigl\|^2 \le 1 - k_2 \,|\eta|^{a-1}\, t^a, \quad t\le \frac{t_2}{|\eta|}.
    \]
    Using Lemma~\ref{lem: technical 2}, this implies that there are $c_2, \tau_2 > 0$ such that
    \[
    \bigr\|\mathrm{e}^{- \bm{C}_\eta t} \bigl\|^2 \le 1 - c_2 t^a, \quad t\le \tau_2
    \]
    for all $\eta \in E$. But then
    \begin{align*}
        \|\bm{T}_E(t)\|^2 = \sup\limits_{\eta \in E} \bigr\|\mathrm{e}^{- \bm{C}_\eta t} \bigl\|^2 \le 1 - c_2 t^a, \quad t\le \tau_2.
    \end{align*}
    To prove \rm (b), use Theorem~4.10 of \cite{achleitner2025hypocoercivity} for an~$\eta_0 \in E$, to obtain
    \[
    1 - c_1\, t^a \le \|\mathrm{e}^{-\bm{C}_{\eta_0}t} \|^2, \quad t \le \tau_1
    \]
    for some $c_1, \tau_1 > 0$. Since $\|\bm{T}_E(t)\| = \sup\limits_{\eta \in E}\bigr\|\mathrm{e}^{- \bm{C}_\eta t} \bigl\|$, we obtain the claim.
\end{proof}

Using Propositions~\ref{prop: thm b}, there are $\tau_1, \tau_2, c_1, c_2 > 0$ such that for $\tau := \min (\tau_1, \tau_2)$ we obtain \eqref{short_time}, which finishes the proof of Theorem~\ref{thm:theorem}~(b).

\section{Application to port-Hamiltonian systems}\label{sec: app}

In \cite{roschkowski2025energymethodsdistributedporthamiltonian} it is shown that for distributed-parameter pH systems with boundary dissipation from~\cite{jacob2012linear}, there is no short-time asymptotic~\eqref{eqn: intro short}. Here, we consider a class of pH systems with distributed dissipation, which is based on~\cite{philipp2021minimizing, gernandt2024stability}
\begin{equation}\label{eqn: Hamiltonian}
    \frac{\partial }{\partial t} x(t, \zeta) =\bm{A}x(t,\zeta)=\left( \bm{P}_1 \frac{\partial}{\partial \zeta} - \bm{R}\right)\bm{H}x(t, \zeta)
\end{equation}
in the Hilbert space 
\[
\mathcal{V} = L^2([0,2\pi];\mathbb{C}^n),\quad \left\langle x, y \right\rangle = \int\limits_0^{2\pi} x(\zeta)^* \bm{H} y(\zeta) \, \mathrm{d}\zeta
\]
with solutions $x(t,\cdot)$ in the Sobolev space $H^1([0, 2\pi]; \mathbb{C}^n)$. 
       Furthermore, assume that   
        $\bm{P}_1 \in \mathbb{C}^{n\times n}$ is Hermitian and invertible,  $\bm{R}\in \mathbb{C}^{n\times n}$ is  positive semi-definite, $\bm{H}\in \mathbb{C}^{n\times n}$ is positive definite. Moreover, we assume that the matrix $\bm{C} = (\bm{R} - i \bm{P}_1)\bm{H}\in \mathbb{C}^{n\times n}$ is hypocoercive on the Hilbert space $\mathcal{H} = \mathbb{C}^n$ with respect to the inner product  $\left\langle v, w\right\rangle_{\mathcal{H}} := v^* \bm{H}w$ for all $v, w \in \mathcal{H} = \mathbb{C}^n$ and $\bm{C}$ has index  $m_{\mathrm{HC}}$. In particular, $i\bm{P}_1\bm{H}$ is skew-Hermitian while $\bm{R}\bm{H}$ is Hermitian on  $\mathcal{H}$. Moreover, we assume that~\eqref{eqn: Hamiltonian} has periodic boundary conditions, i.e.\ \begin{align}
        \label{eq:periodic_bc}
        D(\bm{A}) = \{ x \in H^1([0, 2\pi]; \mathbb{C}^n) \,  | \, x(0) = x(2\pi) \}.\end{align}
The periodic boundary condition~\eqref{eq:periodic_bc} ensures that $\bm{A}$ generates a strongly continuous semigroup $(\bm{T}(t))_{t\geq 0}$, see \cite{jacob2012linear}, and more importantly, 
we can orthogonally decompose \begin{align*}
x(t, \zeta) = \sum\limits_{m = -\infty}^\infty z_m(t) \varphi_m(\zeta), \quad \varphi_m(\zeta) = \mathrm{e}^{i m \zeta}\in L^2([0, 2\pi])\end{align*} with $z_m(t) \in \mathbb{C}^n$. Then \eqref{eqn: Hamiltonian} becomes
\begin{equation}\label{eqn: exm}
    \frac{\mathrm{d}}{\mathrm{d}t} z_m(t) = -\bm{C}_m z_m(t),\quad m \in \mathbb{Z},
\end{equation}
where $\bm{C}_{\mathrm{S}} = -i \bm{P}_1 \bm{H}$, $\bm{C}_{\mathrm{H}} = \bm{R}\bm{H}$, $\bm{C}_m = m\, \bm{C}_{\mathrm{S}} + \bm{C}_{\mathrm{H}}$. 
Note that by considering the Hermitian part, the set of equilibria of \eqref{eqn: Hamiltonian} is described by the closed subspace
\begin{align*}
    \mathcal{V}_\infty := \{ x &\in L^2([0, 2\pi]; \mathbb{C}^n)\ |\ x \text{ is constant and } x(\zeta) \in \ker\bm{R}\bm{H} \text{ almost everywhere} \}
\end{align*}
and let $\bm{Q}_\infty$ denote the projection onto this space.

In the following proposition, we  show the exponential convergence towards the subspace of equilibria and the corresponding short-time behavior, see also \cite{gernandt2024stability} for a related result for more general pH systems.
\begin{proposition}
 Consider a pH system~\eqref{eqn: Hamiltonian}, then the following holds:
    \begin{itemize}
        \item [\rm (a)] There are constants $\tau, c_1, c_2 > 0$ and $a = 2 m_{\mathrm{HC}} + 1$ such that
        \begin{equation}
        \label{short-time-pH}
            1 - c_1 t^a \le \|\bm{T}(t) - \bm{Q}_\infty \|^2 \le 1 - c_2 t^a,\ t\le \tau.
        \end{equation}
        \item [\rm (b)] The subspaces $\mathcal{V}_\infty, \mathcal{V}_\infty^\perp \subset \mathcal{V}$ are invariant for $(\bm{T}(t))_{t\ge 0}$ and there are constants $C, \omega > 0$ such that \[
        \|\bm{T}(t) - \bm{Q}_\infty \| \le C \mathrm{e}^{-\omega t},\quad  t \geq  0.
        \]
        \item [\rm (c)]  Moreover, with $x=\sum_{m=-\infty}^\infty z_m\varphi_m$ the operator
        \begin{align*}
            \bm{P}x &= \sum\limits_{m \in \mathbb{Z}\setminus\{0\}} \sum\limits_{j = 0}^{m_{\mathrm{HC}}}\frac{1}{m^{2j}} (\bm{C}_m^*)^j \bm{C}_m^j \, z_m\, \varphi_m  + z_0\, \varphi_0
        \end{align*}
        is bounded, self-adjoint and satisfies
        \[
        \langle \bm{A}\, x, \bm{P}\, x\rangle + \langle \bm{P}\, x, \bm{A}\, x\rangle \le - \sigma\|x\|^2, \ x \in D(\bm{A}) \cap \mathcal{V}_\infty^\perp
        \] for some $\sigma > 0$ and $\bm{P} \ge \bm{1}_{\mathcal{V}}$.
    \end{itemize}
\end{proposition}
\begin{proof}
   As a consequence of the modal decomposition, we find that $\mathcal{V}_\infty = \ker \bm{A}$, $\mathcal{X}:=\operatorname{span}(\varphi_0e_1, \dots, \varphi_0e_n)^\perp$,  
   $\mathcal{V}_\infty^\perp$, and  
   are invariant for $(\bm{T}(t))_{t\ge 0}$ where $e_1, \dots, e_n$ denote the canonical basis vectors of $\mathbb{C}^n$.
To prove \rm(a), note that part \rm(b) of Theorem~\ref{thm:theorem} can be applied to
    \begin{equation*}
    \frac{\mathrm{d}}{\mathrm{d}t} z_m(t) = -\bm{C}_m z_m(t),\quad m \in \mathbb{Z}\setminus\{0\}
\end{equation*}
where the parameters $\eta$ correspond to $m$ and the set $E = \mathbb{Z}\setminus \{0\}$.
It follows that the restriction of $\bm{T}(t)$ onto $\mathcal{X}$ is the semigroup $\bm{T}_E(t)$ as in Theorem~\ref{thm:theorem} and satisfies the short-time estimate \eqref{short_time}.
Now take $y \in \mathcal{X}^\perp \cap (\mathcal{V}_\infty)^\perp$ and compute 
\begin{equation}\label{eqn: lyapunov constant}
\begin{split}
    \langle \bm{A} y, y\rangle + \langle y, \bm{A}y\rangle \le -2 \kappa \|y\|^2
\end{split}
\end{equation}
where $\kappa > 0$ is chosen so that $\left\langle v, \bm{C}_{\mathrm{H}}v\right\rangle_{\mathcal{H}} \ge \kappa \|v\|^2_{\mathcal{H}}$ for $v \in \mathcal{H} \cap {\ker \bm{C}_{\mathrm{H}}}^\perp$.
This directly implies\begin{equation*}
    \| x(t) - \bm{Q}_\infty x_0\|^2 \le \mathrm{e}^{-\kappa t}\,  \| x_0 - \bm{Q}_\infty x_0\|^2 
\end{equation*}
for all solutions $x \in \mathcal{X}^\perp$ with $x(0) = x_0 \in \mathcal{X}^\perp$.

We may choose $0<\tau\le\tilde{\tau}$ so small that $\mathrm{e}^{-\kappa t} \le 1 - c_2\, t^a$ which yields
\begin{equation}\label{eqn: constant}
    \|x(t) - \bm{Q}_\infty x_0 \|^2 \le (1 -  c_2 t^a) \, \|x_0 - \bm{Q}_\infty x_0\|^2
\end{equation}
for $0\le t\le \tau$.
Taking a solution $x(t)  = x_1(t) + x_2(t)$, where $x_1(t) \in \mathcal{X} \cap D(\bm{A})$, $x_2(t) \in \mathcal{X}^\perp \cap D(\bm{A})$ of \eqref{eqn: Hamiltonian} and $x(0) = x_0$, we find for $t \le \tau$ that
\begin{equation*}
    \begin{split}
        \|x(t) - \bm{Q}_\infty x(t)\|^2 & =\|x_1(t) \|^2 + \left\|x_2(t) - \bm{Q}_\infty x_0 \right\|^2 \le 1 - c_2 t^a, \quad 
    \end{split}
\end{equation*}
where we combined  \eqref{short_time} and \eqref{eqn: constant}. The corresponding lower bound
\[
1 - c_1 t^a \le \|\bm{T}(t) - \bm{Q}_\infty \|^2, \quad t \le \tau
\]
follows from \eqref{short_time}. To prove (c), fix some $y \in D(\bm{A}) \cap \mathcal{V}_\infty^\perp$ and decompose $y = y_1 + y_2$ with $y_1 \in \mathcal{X}$, $y_2 \in \mathcal{X}^\perp\cap \mathcal{V}_\infty^\perp$. It then follows that
\begin{align}\label{eqn: thm long used}
    &\langle \bm{A}\, y_1, \bm{P}\, y_1\rangle + \langle \bm{P}\, y_1, \bm{A}\, y_1\rangle \le -\tilde{\sigma} \|y_1\|^2 
\end{align}
for some $\tilde{\sigma} > 0$
by item (a) of Theorem~\ref{thm:theorem}. Combining \eqref{eqn: lyapunov constant} with $y_2$ plugged in for $y$ and \eqref{eqn: thm long used}, it follows that
\[
        \langle \bm{A}\, y_2, \bm{P}\, y_2\rangle + \langle \bm{P}\, y_2, \bm{A}\, y_2\rangle \le - \sigma\|y\|^2
        \]
for $\sigma = \min(\tilde{\sigma}, 2\, \kappa)$. Moreover, $\bm{P}$ is bounded, self-adjoint and satisfies $\bm{P} \ge \bm{1}_{\mathcal{X}}$ by the same proposition. Since $\bm{P}|_{\mathcal{X}^\perp} = \bm{1}_{\mathcal{X}^\perp}$, it follows that $\bm{P} \ge \bm{1}_{\mathcal{V}}$.
\end{proof}

\section{Conclusion}
In this paper, we study the long- and short-time behavior of semigroups generated by a class of unbounded operators admitting a direct sum decomposition that might be obtained from modal decomposition. Our approach allows for an arbitrary hypocoercivity index of the generating operator and does not require a long-time bound to derive the short-time bound. We applied our results to a class of port-Hamiltonian (pH) systems with distributed damping.  
Future work will include long- and short-time analysis of distributed pH systems under more complicated boundary conditions.

\bibliographystyle{plain}
\bibliography{references}      

\end{document}